\documentclass[amsart,11pt]{article}

\usepackage{amsfonts}
\usepackage{amsmath}
\usepackage{amssymb}
\usepackage{amsthm}
\usepackage{bbm}
\usepackage{graphicx}
\usepackage{gensymb}
\usepackage{comment}
\usepackage{placeins}
\usepackage{geometry}
\usepackage{mathtools}
\usepackage{hyperref}
\usepackage{color}
\usepackage[normalem]{ulem}
\newcommand{\R}{\mathbb R}

\newtheorem{theorem}{Theorem}

\newtheorem{lemma}{Lemma}
\newtheorem{proposition}{Proposition}
\newtheorem{definition}{Definition}

\author{Toby Sanders \\ \small{School of Mathematical and Statistical Sciences, Arizona State University, Tempe, AZ, USA.} }
\date{}
\title{Parameter Selection for HOTV Regularization}

\begin{document}
\maketitle

\begin{abstract}
Popular methods for finding regularized solutions to inverse problems include sparsity promoting $\ell_1$ regularization techniques, one in particular which is the well known total variation (TV) regularization.  More recently, several higher order (HO) methods similar to TV have been proposed, which we generally refer to as HOTV methods. In this letter, we investigate the problem of the often debated selection of $\lambda$, the parameter used to carefully balance the interplay between data fitting and regularization terms.  We theoretically for a scaling of the operators for a uniform parameter selection for all orders of HOTV regularization.  In particular, parameter selection for all orders of HOTV may be determined by scaling an initial parameter for TV, which the imaging community may be more familiar with.  We also provide several numerical results which justify our theoretical findings.
\end{abstract}

\subsubsection*{Keywords}{Image reconstruction, Regularization, Inverse problems, Higher order TV, Parameter selection}

\section{Introduction}
Inverse and imaging problems play a vital role in practically all areas of technology and science. These problems naturally turn up in applications\footnote{MRI, X-ray tomography, optical coherence tomography, and synthetic aperture radar, just to scratch the surface.}, whenever data measurements are collected that must be processed or \emph{inverted} to form resulting images that explain the measurements.
A popular formulation for inverse signal and imaging problems takes the form
\begin{equation}\label{general-reg}
 \min_f F(f) + H(f),
\end{equation}
where $f$ is the image we want to recover, $F$ is a data fitting term, and $H$ is a regularization term that encourages some prescribed smoothness on the image.

If $b$ is the known noisy data, and an estimated forward operator $A$ is known that maps $f$ to the data, then the data fitting term is commonly of the form
\begin{equation}\label{data-fit}
F(f) = \frac{\lambda}{2} \| A(f) - b \|_2^2.
\end{equation}
Typically, $A$ is a linear operator that can be approximated by a matrix, and $\lambda>0$ is a model parameter which is chosen to achieve the best results by properly balancing $F$ and $H$.  

The regularization term $H$ is usually a convex norm or semi-norm.  Historically $H$ has been selected as the $\ell_2$ energy norm of the image to improve ill-conditioning of some inverse problems, a famous method which goes by the name of Tikhonov regularization.  Our interests in this paper are more recently popular $\ell_1$ regularization methods which are often categorized as compressed sensing \cite{CSincoherence}.  These methods promote sparsity of $f$ in some selected domain by setting the regularization term to 
\begin{equation}\label{reg-norm}
H(f) = \| {Tf} \|_1 , 
\end{equation}
where $T$ is the mapping under which the solution is assumed to be sparse.  A very common regularization choice for imaging problems is to use the TV norm, which is equivalent to setting $T$ as the first order finite difference operator that maps $f$ to the difference between all adjacent pixels of $f$.  This type of regularization encourages sparsity in the boundaries or number of jumps in the image \cite{rudin1992nonlinear,strong2003edge}.  Also commonly used are wavelet frames and bases \cite{mallat2008wavelet,starck2010sparse}, and more recently higher order finite difference transforms, an approach sometimes referred to as higher order TV \cite{HOTV,Archibald2015}.   We now address the main focus of this paper, the famous and often debated problem of the selection of the parameter $\lambda$ in (\ref{data-fit}), particularly for higher order TV methods.

With our formulation, it is most generally accepted that the parameter $\lambda$ is chosen to be smaller for noisier data $b$, and chosen to be larger for more accurate data values.  With regards to TV regularization, numerous publications have already been devoted to parameter selection \cite{PS-1,Aujol2006,PS-2}, and many working in the imaging community may be familiar with this task.   Indeed, the motivation behind this work was to generalize the parameter selection based off of TV.  Having such a parameter selection may prove to be useful in many practical applications.  With general image reconstruction problems, it may not be known exactly which regularization method is appropriate prior to searching for a solution, i.e. precise characterization of the underlying image structure may not be known prior to reconstruction.  Therefore several regularization approaches may be carried out in conjunction to realize the most effective approach \cite{SGP-ET}.  Having an accurate parameter scaling between the approaches is necessary for an unbiased comparison of the various approaches and limits any additional parameter tuning.

The referenced approaches for TV parameter selection resolve to iterative selection of $\lambda$ by posterior estimates, or by a priori estimation of noise levels.  Another alternative for general regularization problems is the popular L-curve method \cite{Hansen-l-curve,Calvetti2000423}.  This method is essentially a brute force approach, in which resulting values for $F$ and $H$ in (\ref{general-reg}) are plotted together for an array of parameter selections in order to determine the parameter value that yields the optimal payoff between the $F$ and $H$.   This approach has only been formally studied for Tikhonov regularization, and more importantly it requires one compute many solutions to the problem.

In this paper, we work towards uniform parameter selection for the more general $\ell_1$ methods that we refer to as higher order total variation (HOTV).  A number of HOTV type of approaches have been proposed \cite{HOTV,TGV,papafitsoros2014combined,hu2012higher}.  While each of these methods have subtle differences, the motivation behind each is somewhat similar: to allow for more general function behavior and reduce staircasing effects observed with TV by removing penalties for low order polynomials.  This order is usually flexible within the approach, and the removal of such penalties may be achieved by defining the transform $T$ to in some way approximate higher order derivatives.  For instance, in \cite{TGV} the authors design a regularization term in the continuous setting based on higher order weak derivates and show that low order polynomials are not penalized, again depending on the order of the method.  In \cite{HOTV}, a fourth order PDE-based technique was used to a similar effect, also removing penalties for low order polynomials.  In \cite{papafitsoros2014combined} the authors used within the regularization terms both first order derivatives with second order derivates and demonstrated reduced staircasing effects.  The algorithmic design in each of these approaches are somewhat different depending on the date of publication, but at the current time the alternating direction method of multipliers approach has primarily been adopted \cite{Li2013,tosserams2006augmented,boyd2011alternating}.

We work in particular with the HOTV formulation that has been referred to as polynomial annihilation (PA) regularization \cite{Archibald2015,Archibald:2005:PFE:1061182.1068411}, which takes a straight forward approach to removing penalties from low order polynomials.  This is achieved by using the common discretized approximation of derivatives otherwise known as finite differences.  The MATLAB algorithm for this approach can be downloaded with a corresponding user's guide online \cite{toby-web}.  Compared with other HOTV approaches, which have primarily been demonstrated for denoising and deconvolution problems, our algorithm has been shown effective for a variety of 2D and 3D inverse problems.  In \cite{SGP-ET} it was used to reconstruct 3D nanoscale volumes with dimensions $2^{28}$ total voxels from undersampled tomographic data of total dimension approximately $2^{26}$.  In \cite{sanders2017composite} the approach was adapted for a joint sparsity model to reconstruct scenes from nonuniform Fourier related synthetic aperture radar data sets. 

The ability to effectively solve such large and complex imaging problems is facilitated by a robust algorithm with simple parameter selections.   This work details some important results that has enabled the ease of use of these algorithms, and our numerical studies here additionally confirm the use of these results.  We first argue analytically in two distinct approaches that precise scaling of $\lambda$ depends on the order of the HOTV method used, and the result essentially tells us that we should scale by powers of two\footnote{Although in this work we argue for a scaling of the parameter, the approach has been so successful that the scaling of the parameter has now been removed as an option in the algorithm, and the operators are scaled instead.  Therefore the user may simply choose parameter values without concern for the order of the method.}.  While a simple scaling by two may seem elementary, the details that arrive at this result are nontrivial, and the approach has been extremely effective \cite{SGP-ET}.  We provide robust numerical results that confirm our choice of the scaling of $\lambda$, by both imaging examples and a large series of simulations where we precisely determine the optimal parameters.  In section 2 we introduce our HOTV regularization method.  In section 3 we provide our analysis for appropriate parameter selection, and in section 4 we present several numerical experiments that confirm our proposed parameter selection.

\section{Polynomial Annihilation for $\ell_1$ Regularization}
The linear transform $T$ in (\ref{reg-norm}) has more recently gained attention as a higher order finite difference operator, and in particular one known as polynomial annihilation (PA) regularization \cite{Archibald2015,Archibald:2005:PFE:1061182.1068411,VOTV}.  Loosely speaking, the transform for PA annihilates $f$ at grid points where $f$ is essentially a polynomial of degree less than a desired value, which is called the order of the PA operator.  We denote the PA operator of order $k$ by $T_k$, while noting that $T_1$ is equivalent to TV. As with TV, the PA regularization encourages sparsity in the number jumps in the image.  However, in contrast with TV, characterization of the smooth regions is relaxed, where they are modeled as low order polynomials (e.g. degree 2 or 3) rather than constant functions.

Without formal details, we state that on an equally spaced grid the PA operator is equivalent to a kth order finite difference. Hence, the PA transform has the following definition:

\begin{definition}
Let $ f \in \R^N$ and let $j\le N - k$.  Then the $j^{th}$ element of the order $k$ PA transform of $f$ is given by
\begin{equation}\label{finite-diff}
 ({T}_k {f})_j = \sum_{m=0}^k (-1)^{k+m}  {k \choose m} f_{j+m} .
\end{equation}
\end{definition}

For indices $j>N-k$, we may also define $(T_k f)_j$ using a periodic extension of $f$.  

For higher order finite differences in multiple dimensions, we simply extend $T_k$ to evaluate the difference along each dimension.  For instance, in 2D, we may have $T_k = [ T_k^x , T_k^y]^T$, where $T_k^x$ computes the differences along the $x$ dimension and similarly for $y$.  For example, for $f\in R^{n\times n}$, $$(T_k^x f)_{j,\ell} = \sum_{m=0}^k (-1)^{k+m} {k\choose m} f_{j+m, \ell}.$$

To briefly see that ${T_k}$ defined in (\ref{finite-diff}) indeed annihilates polynomials of degree less than $k$, one can first consider a function defined on a continuous domain.  Here, the operator is evidently the $k^{th}$ derivative operator.  Then discretization of $f$ along with the mean value theorem leads one to find that standard finite differences yield the discretized version of the PA operator, which is well known to be (\ref{finite-diff}) (see for instance, \cite{FDbook}).

Finally, the PA regularization method  can be expressed as finding the solution to
\begin{equation}\label{gen-PA}
\min_f  \Big\{ \frac{\lambda}{2}\| Af - b \|_2^2 +  \| {T_k f} \|_1 \Big\} ,
\end{equation}
where $T_k$ is defined by (\ref{finite-diff}).

 In order to solve (\ref{gen-PA}) for large imaging problems, an efficient iterative algorithm is needed. In our work, we have followed the alternating direction method of multipliers (ADMM) approach in \cite{Li2013}, which rewrites the problem with an augmented Lagrangian function and uses alternating direction minimization to solve equivalent subproblems.  Our algorithm is available online \cite{toby-web}.

\section{Selection of $\lambda$}\label{lambda-select}
In this section we address the appropriate selection of the parameter $\lambda$ in (\ref{gen-PA}), showing explicitly how the parameter depends on the order $k$ of $T_k$.  In general, when solving a problem such as (\ref{gen-PA}), appropriate selection of $\lambda$ is greatly eased by rescaling the operators $A$ and $T_k$, in order to reduce the dependence on the specific problem and regularization transform.  In doing so, $\lambda$ is then primarily dependent on the noise levels and error in the forward model.  \footnote{We note that it is important to make the distinction between noise levels and model errors, since the latter is typically coherent and cannot be captured by standard noise models.  Furthermore, in some cases this means $\|Af -b\|_2$ can not be made arbitrarily small, and thus $\lambda$ should be chosen accordingly.}

To reduce problem dependence, $A$ is rescaled so that $\|A\|_2=1$, and the data values are rescaled accordingly.  Determining this rescaling factor can be accomplished relatively easily, since for $A\in \R^{m\times n}$, $\|A\|_2^2 = \rho(A^T A)$, the largest eigenvalue of $A^T A$ (see, for instance \cite{Golub}).  This eigenvalue can be computed iteratively in computational time that is insignificant to the time needed to solve (\ref{gen-PA}), making this rescaling a practical rule.

Similarly, we may rescale $T_k$ so that $\|T_k\|_1=1$.  However, for simplicity we keep the integer values defined in the operator given by (\ref{finite-diff}), and thus instead choose to rescale $\lambda$ depending on $\|T_k\|_1$.  To do this, we start with the following proposition:
\begin{proposition}
 For the operator $T_k$ defined in (\ref{finite-diff}), the $\ell_1$ norm of $T_k$ is given by
 \begin{equation}\label{Tk-norm1}
  \|T_k \|_1 = 2^k.
 \end{equation}
\end{proposition}
While the exact details of the proof are left to the reader, the general idea is outlined  below.  First recall that the $\ell_1$ norm of a matrix is equivalent to the maximum of the $\ell_1$ norm of the columns, \cite{Golub}. Because the operator is circulant \footnote{In our definition, the operator is only truly circulant if the periodic extension is used.  Nevertheless, the resulting norm is the same in both cases.}, the sum of absolute values of each column of $T_k$ are equivalent, leading to
\begin{equation}\label{Tk-norm}
 ||T_k||_1 = \max_{1 \leq j \leq n} \sum_{i=1}^N \left| (T_k)_{ij}\right| = \sum_{i=0}^k \left| (-1)^{k+i}{k\choose i} \right| = 2^k.
\end{equation}
Thus by definition of the matrix norm (see again, \cite{Golub}) we have
\begin{equation}
\label{eq:T_k2k}
||T_kf||_1 \le 2^k ||f||_1,
\end{equation}
suggesting that the regularization parameter in (\ref{gen-PA}) using $T_k$ be of the form
\begin{equation}
\label{eq:lambdaselection}
\lambda = 2^{k-1} \lambda_1,
\end{equation}
for some appropriate choice of $\lambda_1$ for TV.  We emphasis again here that numerous publications have already been devoted to parameter selection for TV \cite{PS-1,Aujol2006,PS-2}, making a scaling factor for HOTV based off a selection for TV a reasonable approach.  A precise scaling is also useful if one wants to make unbiased comparisons between any two of the approaches.

It is important to note that equality may only be achieved in (\ref{eq:T_k2k}) for signals that are not typically of interest.  For example, ${\bf f} = \{(-1)^j\}_{j = 1}^N$ represents such a signal, but is of limited practical use.  Indeed we are interested in a smaller class of \emph{smooth} signals for which equality may not hold.  Nevertheless, as is shown in what follows, choosing $\lambda$ using (\ref{eq:lambdaselection}) is robust and removes the need for additional parameter tuning for each $k$ used.

The following is our main result, that along with (\ref{Tk-norm1}) also justifies our scaling.

\begin{theorem}\label{main-result}
 Let $f$ be a piecewise constant signal in $\R^N$ with jump locations at the indices over the set defined by $S = \{j~|~ f_{j+1} - f_j \neq 0\}.$  If the jump locations are sufficiently spaced so that for any $~j_1, j_2 \in S$, we have $|j_1 - j_2|>k$, then the following relation between the TV and PA norms holds:
 \begin{equation}
  \|T_k f\|_1 = 2^{k-1} \|T_1 f \|_1.
 \end{equation}

\end{theorem}

Theorem \ref{main-result} is consistent with  (\ref{eq:lambdaselection}) for choosing $\lambda$ in (\ref{gen-PA}), and this scaling is further justified by our numerical results.  Although the statement is for piecewise constant functions, in general the main contribution of to $\|T_k f\|_1$ is at the jumps, even when the function is not a piecewise constant.  We believe a similar argument can be made for higher order piecewise polynomials, but this is notably more difficult and is reserved for potentially later work.

We now prove Theorem \ref{main-result}, which will follow almost immediately by the Lemma 1 below.  We will first need the following proposition:

\begin{proposition}
For all positive integers $k$ and $m$ with $k>m$, we have
\begin{equation}\label{prop1}
 {k-1 \choose m} = \sum_{j=0}^m  {k \choose j} (-1)^{j+m}.
\end{equation}
\end{proposition}
 This identity can be proven with repeated application of the identity 
 $
 {k-1 \choose m} = {k\choose m} - {k-1 \choose m-1}, 
 $
starting with the left hand side until the complete alternating summation arises.  

To prove Theorem 1, we only need to consider a simple example from the class of general piecewise constant functions defined in the theorem.  Let the smooth signal $f$ be defined by
\begin{equation}\label{example-signal}
f_j = 
\begin{cases}
 \alpha &\mbox{if } j\le i\\
 \beta &\mbox{if } j> i.
\end{cases}
\end{equation}

{Here $\alpha \ne \beta$ are constants, and $i$ is some arbitrary index sufficiently far from the boundary.  The following lemma illustrates that $\| T_k f \|_1$ for $f$ in (\ref{example-signal}) increases by a factor of $2$ for each $k$.}
\begin{lemma}
\label{prop:tknorm}
 The discrete signal given by (\ref{example-signal}) satisfies
 \begin{equation}\label{l1-signal}
  \| T_k f \|_1 = 2^{k-1}|\alpha-\beta|.
 \end{equation}

\end{lemma}
\begin{proof}
First observe the evaluation of $(T_k f)_m$ vanishes at points away from the index $i$.  In particular, $(T_k f)_m = 0$ for $m=1,2,\dots , i-k, i+1,i+2,\dots, N-k.$

However, at index $m=i-\ell$, with $0\le \ell <k$, we have
\begin{align*}
(T_k f)_m 
& =  \alpha \left(\sum_{j=0}^{\ell} (-1)^{j+k}{k\choose j}\right)
+ \beta \left( \sum_{j=\ell+1}^k (-1)^{j+k} {k \choose j} \right) \\
& =  (\alpha - \beta) \sum_{j=0}^{\ell} (-1)^{j+k}{k\choose j},
\end{align*}
where we have used the fact that $$\sum_{j=0}^{\ell} (-1)^{j+k}{k\choose j} = - \sum_{j=\ell+1}^k (-1)^{j+k} {k \choose j}.$$  Applying (\ref{prop1})  yields
$
(T_k f)_m = (\alpha - \beta) {k-1 \choose \ell}(-1)^{k-\ell}.
$
Finally, taking the $\ell_1$ norm we obtain
\begin{equation}
 \| T_k f\|_1 = |\alpha - \beta| \sum_{\ell = 0}^{k-1} {k-1 \choose \ell} = 2^{k-1}|\alpha - \beta|.
\end{equation}
\end{proof}

Theorem \ref{main-result} now follows immediately from Lemma \ref{prop:tknorm} by generalizing this discrete signal in (\ref{example-signal}) to a piecewise constant signal with sufficiently spaced jump locations.  We also note that this result generalizes without any trouble to two and three-dimensional signals.

\section{Numerical Experiments}

\subsection{Numerical Determination of Optimal $\lambda$}

 \begin{figure}
 \centering
 \includegraphics[scale=.8]{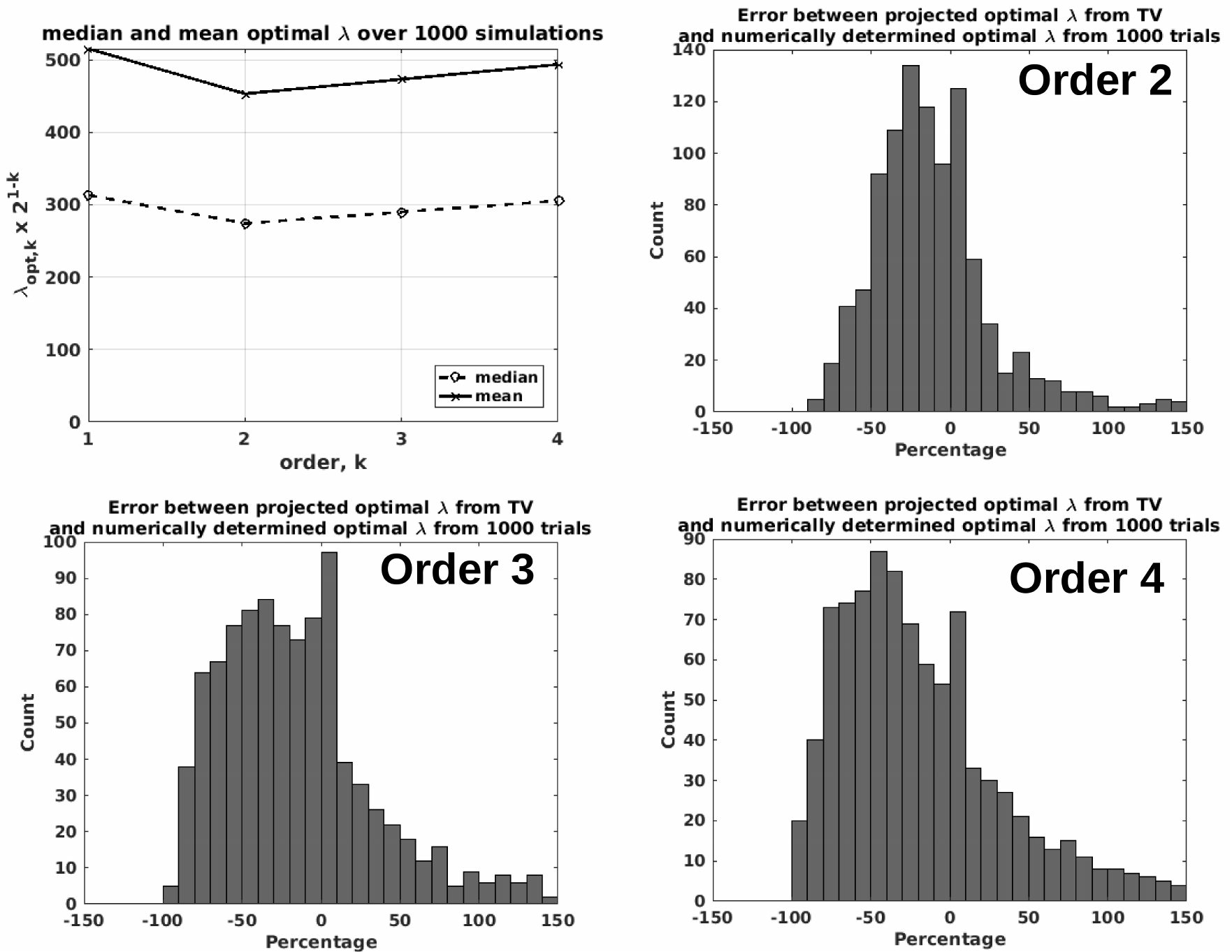}
 \caption{Results for optimal $\lambda$ determined from 1000 simulations.  For each order (2,3,4), the histograms show the errors between the projected optimal $\lambda$ from TV as given by (\ref{eq:lambdaselection}) and the numerically determined optimal $\lambda$ determined by (\ref{opt-lambda}).  Shown in the top left panel are the mean and median optimal $\lambda$ values from all 1,000 simulations for each order (the values are scaled by $2^{1-k}$ for each order $k$).}
 \label{fig3}
\end{figure}
To determine the exactness of our scaling in a robust numerical fashion we set up 1,000 1-D numerical simulations, in which the optimal $\lambda$ is determined for each simulation.  The simulations were each uniquely set up in the following way:  
\begin{itemize}
\item \textbf{Signal construction:}  In each simulation, a piecewise polynomial signal was randomly generated over a 256 point mesh.  Each signal was generated by first randomly designating some number of jumps for the signal, between 2 and 20 jumps with each possible number having equal probability.  Then the location for each jump was randomly determined.  Over each continuous section between the jumps, a polynomial of either order 0, 1, or 2 is generated for the each section, where again this order is randomly chosen and each order has equal probability.  

\item \textbf{Sampling:} A random sampling matrix $A$ was generated to sample the generated signal, where the sampling rate is randomly chosen between 25\% and 100\%.  The sampling rate is defined by the number of samples, or rows of $A$, divided by the number of grid points, or columns of $A$.  The entries of $A$ are determined by randomly choosing 10\% of the entries to be nonzero, where the nonzero entries are uniformly distributed over the interval $[0,1]$.  

\item \textbf{Noise:} Mean zero Gaussian white noise was added to the data samples, where the standard deviation of the noise is set to be a random value between 0 and 3, each value having equal probability of occurring.  
\end{itemize}

This set up gives us an extremely wide variety of examples to draw conclusions and eliminates any possible biases in the particular set up.

Once the problem is set up, the optimal $\lambda$ is determined by a brute force search, which is terminated once the optimal value is determined within a sufficient tolerance.  We choose the optimal $\lambda$ to be defined as
\begin{equation}\label{opt-lambda}
 \lambda_{opt,k} = \arg\min_\lambda \Big\{ || f_{true} - f^*\|_2 ~ \Big| ~ f^* = \arg\min_f \frac{\lambda}{2} \|Af-b\|_2^2 + \|T_kf\|_1\Big\},
\end{equation}
where $f_{true}$ is the original signal.  In other words, the optimal $\lambda$ is that which yields a solution closest to the original function.

According to our work, ideally the optimal lambda for order $k$ is given by $\lambda_{opt,k} = 2^{k-1}\cdot \lambda_{opt,1}$.  Therefore we define the relative error of our approach by 
\begin{equation}\label{lambda-error}
 \frac{2^{1-k}\cdot \lambda_{opt,k} - \lambda_{opt,1}}{\lambda_{opt,1}},
\end{equation}
which is calculated for all 1000 simulations and for $k=2,3,4$.

The results from the simulations are presented in Figure \ref{fig3}.  For each order, the frequency of each relative error (as a percentage) is shown in histogram plot.  Each shows a small offset centered curve to the left of 0\%, while also showing a secondary peak directly at 0\%.  There is a skew to the right however.  This pulls the overall average optimal $\lambda$ to the right of the main peaks.  The mean and median optimal $\lambda$ is plotted in the top left panel, where for each order the values are rescaled by $2^{1-k}$, so that by our scaling of $\lambda$ we would ideally see straight line for these two plots.  There is a modest dip for orders 2 and 3, but not significantly so.  Generally speaking these results confirm that our scaling yields fairly accurate results.  However, the histogram plots show wider curves for higher orders, suggesting that the scaling is more often accurate for smaller orders and less reliable for larger orders.

\subsection{Two-dimensional Imaging Experiments}
For two-dimensional numerical experiments we reconstruct two phantom images of size $256\times256$ from undersampled and noisy data.  The phantoms are shown in Figures \ref{fig1} and \ref{fig2} (a).  The first phantom is the classical Shepp-Logan phantom, which is piecewise continuous and well suited for TV.  The second phantom is similar to a cross-sectional brain image and has many nontrivial image features making it better suited for HOTV methods.

The sampling of the images was at a rate of 50\%, giving $256^2/2$ total data points, which are corrupted with Gaussian white noise yielding an SNR of 23.75.  The sampling matrices are generated just as in the previous 1-D set up.

Reconstructions from the data were performed with HOTV regularization of order 1 (TV), 2, 3, and 4, both with and without the suggested scaling of $\lambda$ based off an initial selection for the TV reconstruction.  Some of the resulting reconstructed images are shown in Figure 1 and 2 (b)-(f).  The least squares solution in (b) is shown for baseline comparison and to illustrate the difficulty of reconstructing in this setting.  For the Shepp-Logan phantom, some of the fine features present in (c-e) with the scaling of $\lambda$ become absent in (f) where the upscaling of $\lambda$ is not used.  Particularly the three small ellipses at the bottom of each become difficult to make out when the scaling is not used.

\begin{figure}[ht!]
 \centering
 \includegraphics[scale=.92]{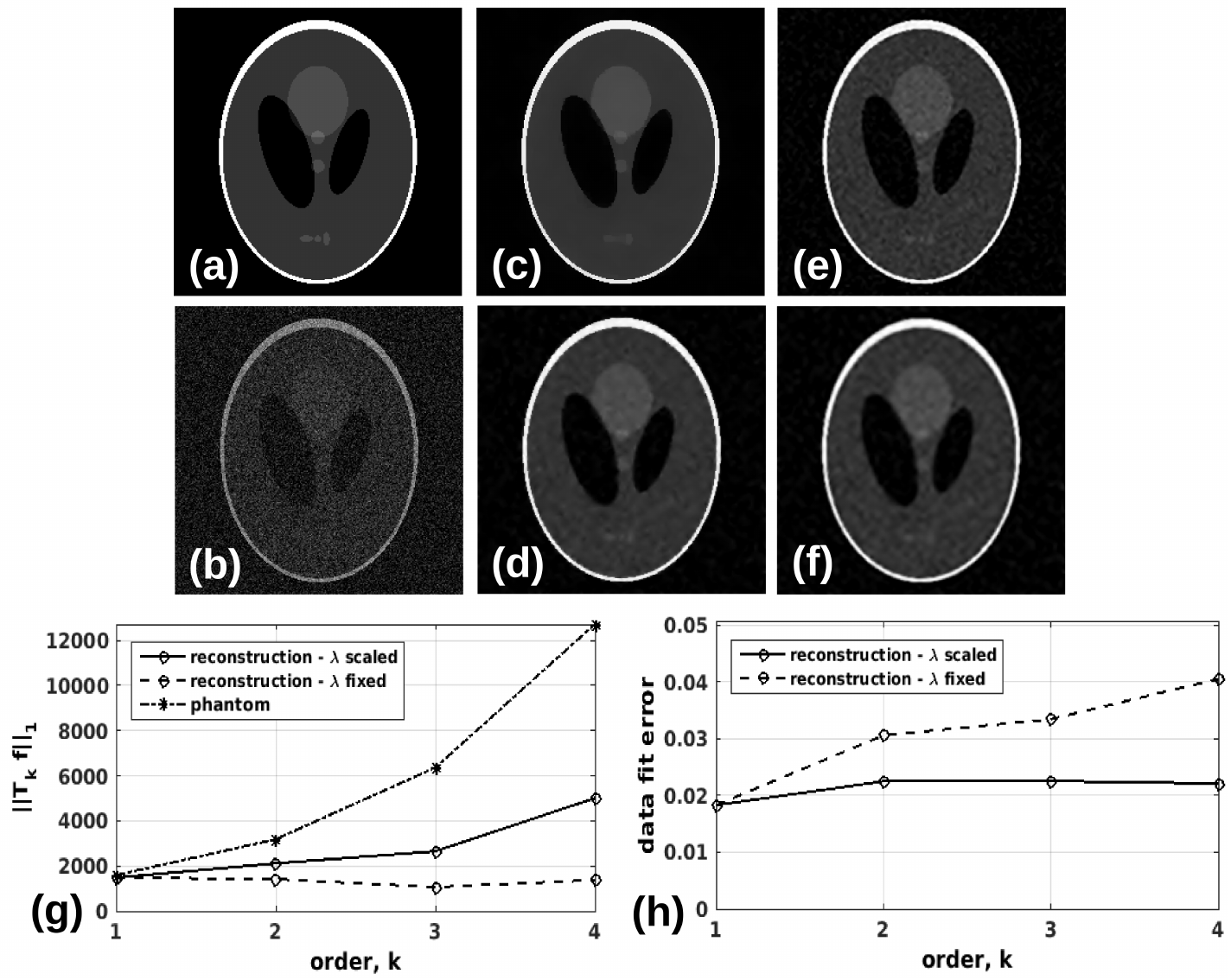}
 \caption{(a) Shepp-Logan phantom.  (b) Least squares reconstruction.  (c) TV reconstruction.  (d-e) HOTV reconstructions of orders 2 and 3 respectively with $\lambda$ scaled according to (\ref{eq:lambdaselection}). (f) HOTV order 3 reconstruction without scaling $\lambda$.  (g) HOTV norms for each reconstructions compared with the phantom.  (h) Relative data fitting error for each reconstruction.}
 \label{fig1}
\end{figure}

For the second phantom, the TV model smooths over the fine details of the image.  A comparison with the order 3 approach with the scaling (e) and without the scaling (f) are again made.  These again indicate that $\lambda$ is too small in (f), and some of the accurate image features that are seen in (e) are smoothed together in (f).

Perhaps most revealing are the plots in Figure 1 (g) and (h).  In (g), the corresponding PA $\,\ell_1$ norms, $\|T_k f\|_1$, for each reconstruction are computed and compared with each corresponding norm for the phantom image.  First note that the norms for the phantom closely follow our theoretical result in Theorem 1 of an increase by a factor of two for each increase in the order.  Without the scaling of $\lambda$, the resulting norms of the reconstructions essentially remain flat, where we already observed the true norm of the phantom increases by a factor of around two.  Thus with the scaling of $\lambda$, the reconstructions follow the trend for the phantom.  In (h), the relative error of the reconstructions is shown, which is defined by $\|Af-b\|_2/\|b\|_2.$  It is clearly shown that for reconstruction without scaling $\lambda$, the relative error increases significantly for each increase in the order.  This indicates that $\lambda$ is too small for these higher order methods.  On the other hand, the reconstructions with the scaled $\lambda$ show comparable relative errors, as one would hope for.

 \begin{figure}
 \centering
 \includegraphics[scale=1]{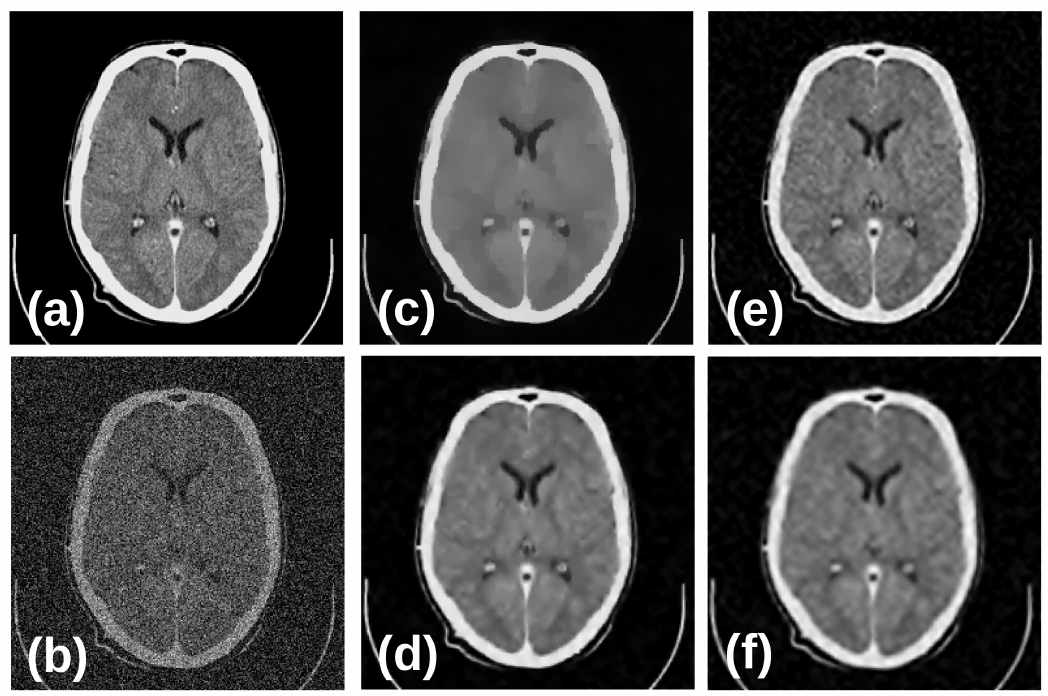}
 \caption{(a) Phantom image.  (b) Least squares reconstruction.  (c) TV reconstruction.  (d-e) HOTV reconstructions of orders 2 and 3 respectively with $\lambda$ scaled according to (\ref{eq:lambdaselection}). (f) HOTV order 3 reconstruction without scaling $\lambda$.}
 \label{fig2}
\end{figure}

\section{Conclusions}
We investigated the common problem of parameter selection for regularized inverse problems, in particular for our successful TV and HOTV methods known as polynomial annihilation (PA).  Our analysis given by Proposition 1 and Theorem 1 showed that for each increase in the order of the method, the parameter $\lambda$ should be increased by a factor of two.  These theoretical results essentially summarize that for a certain class of signals, the $\ell_1$ norm of our HOTV transform of these signals increases by a factor of two for each increase in the order, and thus $\lambda$ should be scaled according to (\ref{eq:lambdaselection}).  Two-dimensional numerical results justified this scaling of $\lambda$, particularly showing that when the scaling was carried out, the data fitting error was nearly even for all orders, and the increase in the $\ell_1$ norms trended with the true phantom norms that agreed with the theoretical results.  More robust numerical analysis was carried out in a set of 1,000 1-D simulations, and showed that the optimal $\lambda$ for each order generally followed the trend of scaling according to (\ref{eq:lambdaselection}) with some modest error.  This scaling allows for unbiased comparisons between the varied approaches and may be especially useful for anyone already familiar with the classical TV approach.

\section*{Acknowledgements}
This work is supported in part by the grants NSF-DMS 1502640 and AFOSR FA9550-15-1-0152.

\end{document}